\documentclass[12pt, oneside]{amsart}

\usepackage{amsmath,amssymb,cite,mathrsfs,tikz-cd}
\usepackage[all]{xy}
\usepackage{typearea} 
\usepackage[backref=page]{hyperref} 
\renewcommand*{\backrefalt}[4]{{\tiny%
    (\ifcase #1 Not cited.%
          \or Cited on p.~#2.%
          \else Cited on pp. #2.%
    \fi%
    )}}
    
\usepackage{color} 

\usepackage{epstopdf} 
\usepackage{booktabs}

\newtheorem{teo}{Theorem}[section]
\newtheorem{thm}[teo]{Theorem}
\newtheorem{prop}[teo]{Proposition}
\newtheorem{lemma}[teo]{Lemma}
\newtheorem{cor}[teo]{Corollary}
\newtheorem{conj}[teo]{Conjecture}

\newtheorem{def-prop}[teo]{Definition-Proposition}

\theoremstyle{definition}

\newtheorem{defn}[teo]{Definition}
\newtheorem{expl}[teo]{Example}
\newtheorem{rmk}[teo]{Remark}

\theoremstyle{plain}

\newcommand{\exampleandremarkendmark}{%
  \begingroup
    \renewcommand{\qedsymbol}{\ensuremath{\diamond}}%
    \qed
  \endgroup
}
\AddToHook{env/eg/end}{\exampleandremarkendmark}
\AddToHook{env/expl/end}{\exampleandremarkendmark}
\AddToHook{env/rmk/end}{\exampleandremarkendmark}
\AddToHook{env/remark/end}{\exampleandremarkendmark}

\newtheoremstyle{named}{}{}{\itshape}{}{\bfseries}{.}{.5em}{\thmnote{#3 }#1}
\theoremstyle{named}

\makeatletter
\newcommand{\neutralize}[1]{\expandafter\let\csname c@#1\endcsname\count@}
\makeatother

\numberwithin{equation}{section}

  \newcommand{\Q}{\mathbb{Q}}

  \newcommand{\Z}{\mathbb{Z}}

\newcommand{\CC}{\mathbb{C}}

\newcommand{\QQ}{\mathbb{Q}}
\newcommand{\RR}{\mathbb{R}}

\newcommand{\ZZ}{\mathbb{Z}}

\newcommand{\calA}{\mathcal{A}}

\newcommand{\calF}{\mathcal{F}}

\newcommand{\calO}{\mathcal{O}}

\newcommand{\id}{\mathrm{id}}

  \renewcommand{\bar}{\overline}
  
  \renewcommand{\hat}{\widehat}

  \providecommand{\frac}[1]{\operatorname{Frac}(#1)}

\DeclareMathOperator{\Pic}{Pic}

\DeclareFontFamily{U}{mathx}{}
\DeclareFontShape{U}{mathx}{m}{n}{<-> mathx10}{}
\DeclareSymbolFont{mathx}{U}{mathx}{m}{n}
\DeclareMathAccent{\widehat}{0}{mathx}{"70}
\DeclareMathAccent{\widecheck}{0}{mathx}{"71}

  \renewcommand{\ker}{\operatorname{ker}}
  
  \renewcommand{\lim}{\operatorname{lim}}
  \renewcommand{\deg}{\operatorname{deg}}

\newcommand{\cA}{\mathcal{A}}

\newcommand{\cF}{\mathcal{F}}

\newcommand{\cV}{\mathcal{V}}
\newcommand{\cW}{\mathcal{W}}
\newcommand{\cX}{\mathcal{X}}
\newcommand{\cY}{\mathcal{Y}}
\newcommand{\cZ}{\mathcal{Z}}

\title{The Beilinson--Bloch heights of generalized Ceresa cycles}
\author{Kaiyuan Gu}
\address{BICMR, Peking University, Haidian District, Beijing 100871, China}
\email{gky@stu.pku.edu.cn}
\date{September 2026}

\begin{document}

\begin{abstract}
    We compute the Beilinson--Bloch height of the generalized Ceresa cycle $\mathrm{Ce}(V)=\frac 1 2(V-[-1]_*V)$ for any cycle $V$ of pure dimension $d$ on an abelian variety $A$, in terms of an arithmetic intersection number of a certain adelic line bundle. As an application, we prove that this height is given by an adelic line bundle over algebraic families.
\end{abstract}

\maketitle
\section{Introduction}
Let $A$ be an abelian variety of dimension $g$ defined over a number field $K$\footnote{The results have analogous formulations over function fields. We restrict ourselves to the number field case.}, and let $V=\sum_i a_i[V_i]$ be any cycle of pure dimension $d$ on $A$, with $a_i\in\Q$ and $V_i$ being integral subvarieties of dimension $d$. Let $L$ be an ample, symmetric line bundle on $A$, with a rigidification along the identity section. Let $m:A\times A\to A$ be the addition map, and $p_1,p_2:A\times A\to A$ be the projections to the factors. Set $Q=Q_L:=m^*L-p_1^* L-p_2^*L\in \Pic(A\times A)$. Denote the Fourier--Mukai transform by 
\begin{align*}\calF_L:\mathrm{Ch^*}(A)_\QQ&\to \mathrm{Ch}^*(A)_\QQ,\\
Z&\mapsto \calF_L(Z)=p_{2,*}(p_1^*(Z)\cdot e^{c_1(Q)}).\end{align*}

We define the (normalized) \textit{generalized Ceresa cycle} by $$\mathrm{Ce}(V)=\frac 1 2(V-[-1]_*V).$$This cycle is homologically trivial, since $[-1]_*$ acts as the identity on $H_{2d}(A(\CC),\QQ)$. We define its Beilinson--Bloch height (of Fourier type) by $$h_{\mathrm{BB},\cF}(\mathrm{Ce}(V))=\frac{(-1)^d}{[K:\Q]}\langle \mathrm{Ce}(V),\cF_L\mathrm{Ce}(V)\rangle_{\mathrm{BB}}.$$ 
For a brief account of the Beilinson--Bloch height pairing, see Section~\ref{sec:heights}. Here the pairing takes the component of $\cF_L\mathrm{Ce}(V)$ of complementary codimension $d+1$, which is cohomologically trivial, and whose admissible extension can be explicitly expressed; see Section~\ref{sec:admissible-fourier}.

The factor $(-1)^d$ is motivated by the arithmetic standard conjecture of Hodge type. We raise the following similar conjecture for the positivity of Beilinson--Bloch heights of Fourier type.
\begin{conj}\label{Conj:FourierPos}
	For any cycle $V$ of pure dimension $d$, we have $$h_{\mathrm{BB},\cF}(\mathrm{Ce}(V))\geq 0.$$
\end{conj}

This conjecture would follow from the arithmetic standard conjecture of Hodge type; see Corollary~\ref{cor:fourier-height-sign}.

For a smooth, projective, geometrically connected curve $C$ of genus $g>1$, let $C=V\subset A=\operatorname{Jac}(C)$ denote any Abel--Jacobi image, and let $L$ represent the principal polarization given by the theta divisor. Zhang's formulae \cite[Theorems~1.5.6 and~2.3.5]{Zhang_GScycle} relate the corresponding heights to Gross--Schoen cycles. With the current conventions, Zhang's result reads as follows:
\begin{thm}[Zhang, curve case]Endow $L$ with its canonical adelic metric. Let $\bar Q=m^*\bar L-p_1^*\bar L-p_2^*\bar L\in \hat\Pic(A\times A)_{\mathrm{int}}$. Then we have
    $$h_{\mathrm{BB},\cF}(\mathrm{Ce}(C))=-\frac 1 {6[K:\Q]}\bar Q^{3}\cdot[C\times C].$$
\end{thm}

The goal of this note is to extend this to an arbitrary cycle $V$ of pure dimension $d$. Indeed, we prove the following formula for general $A,V,L$:

\begin{thm}[Main Theorem]\label{thm:main}
Let $V$ be any cycle of pure dimension $d$ on $A$. Endow $L$ with its canonical adelic metric. Let $\bar Q=m^*\bar L-p_1^*\bar L-p_2^*\bar L\in \hat\Pic(A\times A)_{\mathrm{int}}$. Then we have
    $$h_{\mathrm{BB},\cF}(\mathrm{Ce}(V))=\frac{(-1)^d}{(2d+1)![K:\Q]}\bar Q^{2d+1}\cdot[V\times V].$$
Here, we write $[V\times V]=\sum_{i,j}a_i a_j[V_i\times V_j]$ for the product cycle of pure dimension $2d$ in $A\times A$.
\end{thm}

\begin{expl}
    If $V=\{x\}$ with $x\in A(K)$, by Theorem~\ref{thm:main}, the Beilinson--Bloch height of its Ceresa cycle $\mathrm{Ce}(x)$ is twice of the absolute N\'eron--Tate canonical height of the point: $$h_{\mathrm{BB},\cF}(\mathrm{Ce}(x))=\frac 1 {[K:\Q]}\hat\deg(\bar Q|_{(x,x)})=2\hat h_L(x).$$
\end{expl}

\begin{rmk}
The following Beilinson--Bloch height of Lefschetz type is more commonly seen in the literature. For $1\leq q\leq g$ and a cycle $z\in\mathrm{Ch}^q_1(A)$ of Beauville degree $1$, we define the Beilinson--Bloch height of Lefschetz type by
\[
h_{\mathrm{BB},\mathrm{Lef}}(z)
=\frac{(-1)^q}{[K:\Q]}\langle z,\mathsf L^{g+1-2q}z\rangle_{\mathrm{BB}},
\]
where $\mathsf L$ is the Lefschetz operator $z\mapsto c_1(L)\cdot z$. If $g+1-2q<0$, the power denotes the inverse of the corresponding Lefschetz isomorphism on Beauville degree $1$, as explained in Section~\ref{sec:lefschetz}.

For $2q\leq g+1$ and a primitive cycle $z\in P^q_1(A)$, that is, $\mathsf L^{g+2-2q}z=0$, the comparison in Proposition~\ref{prop:fourier-lefschetz} gives
\[
h_{\mathrm{BB},\cF}(z)
=\frac{h^0(L)}{(g+1-2q)!}\,
h_{\mathrm{BB},\mathrm{Lef}}(z).
\]
However, in this article, the Beilinson--Bloch heights are those of Fourier type, unless otherwise specified.
\end{rmk}

As an application of the theory of Deligne pairing in \cite{YuanZhang_Quasiproj}, we obtain the following:

\begin{cor}
   Let $S$ be an irreducible, normal quasi-projective variety over $\bar \Q$. Let $\cV \subset \cA\to S$ be a flat family of $d$-dimensional closed subvarieties of abelian varieties, with a relatively ample, symmetric line bundle $L$ on $\cA$, rigidified along the identity section. Write $\bar L\in \hat\Pic(\cA/\Z)_{\mathrm{int}}$ for its canonical adelic metrization, and set $\bar Q_L=m^*\bar L-p_1^*\bar L-p_2^*\bar L\in \hat\Pic((\cA\times_S \cA)/\ZZ)_{\mathrm{int}}$. Then the fiberwise Beilinson--Bloch height (with respect to $L$) $$s\in S(\bar \QQ)\longmapsto h_{\mathrm{BB},\cF}(\mathrm{Ce}(\cV_s))$$ is the height function of an integrable adelic $\Q$-line bundle $$\frac{(-1)^d}{(2d+1)!}\langle \bar Q_L^{2d+1}\rangle_{(\cV\times_S \cV)/S}\in \hat\Pic(S/\ZZ)_{\QQ,{\mathrm{int}}}.$$

More generally, for a relative cycle $\cV=\sum_i a_i[\cV_i]$ with $\cV_i\subset\cA$ being flat families of $d$-dimensional subvarieties, bilinearity gives an integrable adelic $\Q$-line bundle
\[
\frac{(-1)^d}{(2d+1)!}\sum_{i,j}a_i a_j
\langle\bar Q_L^{2d+1}\rangle_{(\cV_i\times_S\cV_j)/S},
\]
inducing the height function $s\mapsto h_{\mathrm{BB},\cF}(\mathrm{Ce}(\cV_s))$.
\end{cor}

\begin{expl}\label{expl:translates}
	We give an example on the variation of Beilinson--Bloch heights on the family of cycles given by translations of a fixed cycle $V$.
	
	Let $S=A$, $\cA=A\times S$, and consider the family $\cV_s:=s+V:=(t_s)_*V$ for a fixed cycle $V\in\mathrm{Ch}^{g-d}(A)$, where $t_s(x)=x+s$. Write
\[
h_{\mathrm {BB},V}(s):=h_{\mathrm{BB},\cF}(\mathrm{Ce}(s+V)),\qquad
 a_V:=\frac{(-1)^d}{(2d)!}\deg\bigl(c_1(Q)^{2d}\cdot[V\times V]\bigr).
\]
For $s\in A(F)$, where $F/K$ is finite, let $\bar P_s=\bar Q|_{A_F\times\{s\}}$. Theorem of the cube gives
\[
(t_s\times t_s)^*\bar Q=\bar Q+(p_1^*\bar P_s+p_2^*\bar P_s)+2\pi^*(s^*\bar L),
\]
where $\pi:A_F\times A_F\longrightarrow\operatorname{Spec}F$ is the structural morphism. Consequently, Theorem~\ref{thm:main} yields
\begin{align*}
h_{\mathrm {BB},V}(s)=&h_{\mathrm {BB},V}(0)+2a_V\hat h_L(s)\\
&+\frac{(-1)^d}{(2d)![F:\Q]}
  \bar Q^{2d}\cdot(p_1^*\bar P_s+p_2^*\bar P_s)\cdot[V\times V]\\
&+\frac{(-1)^d}{2(2d-1)![F:\Q]}
  \bar Q^{2d-1}\cdot(p_1^*\bar P_s+p_2^*\bar P_s)^2\cdot[V\times V],
\end{align*}
since every term involving $(p_1^*\bar P_s+p_2^*\bar P_s)^{> 2}$ vanishes. Here the arithmetic intersections are taken over $F$, and the last term is omitted when $d=0$. In the formula, the second term and the last term are quadratic in $s$, while the third term is linear in $s$. Therefore, we see $h_{\mathrm {BB},V}(s)$ is (at most) quadratic on the $\Q$-vector space $S(\bar \Q)_\Q=A(\bar \Q)_\Q$.
\end{expl}

As another application, the admissible Fourier transform, constructed in Section~\ref{sec:admissible-fourier}, gives an explicit formula for admissible extensions of cycles of positive Beauville degree. Recall the Beauville decomposition:
\[
\mathrm{Ch}^q_s(A)
=\{z\in\mathrm{Ch}^q(A):[n]^*z=n^{2q-s}z\text{ for all }n\geq 2\},
\qquad
\mathrm{Ch}^q_{>0}(A)=\bigoplus_{s>0}\mathrm{Ch}^q_s(A).
\]
Here $s$ is called the \textit{Beauville degree}. We prove that every cycle in $\mathrm{Ch}^q_{>0}(A)$ admits an admissible extension. This recovers the results of K\"unnemann \cite{Kunnemann_Height2001} in the special case of positive Beauville degree.
\begin{cor}\label{cor:positive-admissible}
    Let $A$ be a semistable abelian variety of dimension $g$ over a number field $K$.
    Fix a projective regular model $\calA/\calO_K$, which is known to exist by \cite{Kunnemann_ProjRegModel_BBht}.
    Then every class in $\mathrm{Ch}^q_{>0}(A)$, for $1\leq q\leq g$, admits an admissible extension to $\cA$.
\end{cor}

\begin{rmk}
	Indeed, Beauville conjectured in \cite{Beauville1986} that for any $q$ and any $s<0$, $\mathrm{Ch}^q_s(A)=0$, and that for any $q$, the restriction of the cycle class map $\mathrm{cl}: \mathrm{Ch}^q_0(A)\to  H^{2q}(A,\QQ)$ is injective. Under these conjectures, Corollary~\ref{cor:positive-admissible} applies to every homologically trivial Chow class.
\end{rmk}
\subsection{Conventions}
Our conventions for arithmetic(al) schemes, homological and cohomological arithmetic Chow groups are the same as those in \cite{Zhang_GScycle}. Algebraic Chow groups have rational coefficients; for arithmetic Chow groups we follow Zhang's convention allowing real coefficients on vertical cycles.

\subsection{Acknowledgements}
I thank Robin de Jong for his comments on a previous version of this work, which substantially improve the writings of this paper, and motivate the discussion in Section~\ref{sec:lefschetz}. I thank my advisor Xinyi Yuan for helpful discussions. KG is partially supported by Beijing Natural Science Foundation 24QY0003.

\subsection{Disclosure of AI use}
The author came up with the idea on defining the Beilinson--Bloch height of Fourier type for general cycles in abelian varieties in 2024. When he read Zhang's paper, he realized that \cite[Lemma 2.1.2]{Zhang_GScycle} applies when one wants to define the admissible extension of the Fourier--Mukai transform. This part is purely human-generated.

To make his observation precise, he asked AI to search for a reference of the norm on the group numerical equivalence classes of cycles, so that he can control the error term in the limiting process. AI found \cite[Corollary 3.16]{FulgerLehmann_PosConeCycle}, and the author then applied it in the proof of Theorem~\ref{thm:finite-admissibility}. 

After the draft is finished, the author also used AI to fix typos, check grammar and audit the proofs.

\section{Notations}\label{sec:heights}
\subsection{The Beilinson--Bloch pairing}
Let $X/K$ be a smooth, projective variety of dimension $n$, and write $\mathrm{Ch}^q(X)_{\mathrm{hom}}$ for the rational equivalence classes of homologically trivial cycles of codimension $q$. We briefly recall the arithmetic construction of the Beilinson--Bloch pairing \cite{Beilinson_Height,Bloch_Height}; see also \cite[\textsection~2.1, p.~17]{Zhang_GScycle}.

On a projective regular model $\cX/\calO_K$, an \textit{admissible extension} of $z\in\mathrm{Ch}^q(X)_{\mathrm{hom}}$ is an arithmetic class $\hat z=(\widetilde z,g_z)$ extending $z$, such that 
\begin{itemize}
	\item the curvature $\omega(\hat z)=0$ at every archimedean place;
	\item the cycle $\widetilde z$ is numerically trivial after restriction to every irreducible component of every special fiber.
\end{itemize}

For $z\in\mathrm{Ch}^q(X)_{\mathrm{hom}}$ and $w\in\mathrm{Ch}^{n+1-q}(X)_{\mathrm{hom}}$ admitting admissible extensions, put
\[
\langle z,w\rangle_{\mathrm{BB}}
=\hat\deg_{\cX}(\hat z\cdot\hat w),
\]
where $\hat w$ is any arithmetic extension of $w$. Admissibility makes this independent of the chosen extensions; the pairing is also independent of the model. The existence of such extensions is still conjectural in general. In our case, K\"unnemann proved the existence of such extensions for abelian varieties with everywhere semistable reduction in \cite[Theorem~1.6]{Kunnemann_Height2001}. 

Under a finite extension $F/K$, the projection formula gives \cite[\textsection~1.4 and Lemma~1.5]{Kunnemann_Height2001}
\[
\langle z_F,w_F\rangle_{\mathrm{BB},F}
=[F:K]\langle z,w\rangle_{\mathrm{BB},K}.
\]
Therefore, if the admissible extensions are constructed over $F$, the Beilinson--Bloch height pairing over $K$ is recovered by dividing by $[F:K]$. 

\subsection{Adelic line bundles and absolute heights}
We use the theory of \cite{YuanZhang_Quasiproj} for adelic line bundles on quasi-projective varieties. For a quasi-projective variety $X/\bar \Q$, the group of adelic line bundles on $X$ is denoted by $\hat \Pic(X/\Z)$. If $X$ is already projective, we use the more classical notation $\hat \Pic(X)$. 

Choose a number field of definition $K$. For $\bar M\in\hat\Pic(X/\Z)$, $x\in X(\bar K)$, and a finite extension $F/K$ with $x\in X(F)$, the (absolute) height is
\begin{equation*}
h_{\bar M}(x)=\frac{\hat\deg(x^*\bar M)}{[F:\Q]}.
\end{equation*}
This quantity is independent of $F$. If $W\subset X$ is an integral projective subvariety of dimension $r$ defined over $K$, and $\bar M_0,\ldots,\bar M_r$ are integrable adelic line bundles on $X$, we use $\bar M_0\cdots\bar M_r\cdot[W]$ to denote the unnormalized arithmetic intersection number over $K$. For integrable $\bar M$ with $M|_W$ ample, the (normalized) height of $W$ is defined by
\[
h_{\bar M}(W)=\frac{\bar M^{r+1}\cdot[W]}{(r+1)[K:\Q]\deg_M(W)},
\qquad \deg_M(W)=c_1(M)^r\cdot[W].
\]
For a flat projective family $f:\cX\to S$ of relative dimension $r$ over a normal quasi-projective base scheme $S/\bar \Q$, and integrable adelic line bundles $\bar M_0,\ldots,\bar M_r\in \hat \Pic(\cX/\Z)_{\mathrm{int}}$, compatibility of the Deligne pairing with base change gives
\[
h_{\langle\bar M_0,\ldots,\bar M_r\rangle_{\cX/S}}(s)
=\frac{\bar M_0|_{\cX_s}\cdots\bar M_r|_{\cX_s}\cdot[\cX_s]}{[F:\Q]},
\]
whenever the family is defined over $F$ and the point $s\in S(F)$.

\section{The admissible Fourier transform}\label{sec:admissible-fourier}
Let $$T_{\mathrm{Ce}(V)}=\mathrm{pr}_{d+1}(\cF_L(\mathrm{Ce}(V)))=\frac 1{(2d+1)!}p_{2,*}\bigl(p_1^*(\mathrm{Ce}(V))\cdot c_1(Q)^{2d+1}\bigr).$$
Then, by the Beauville decomposition of $\mathrm{Ch}^*(A)$, we find that $$h_{\mathrm{BB},\cF}(\mathrm{Ce}(V))=\frac{(-1)^d}{[K:\Q]}\langle\mathrm{Ce}(V),T_{\mathrm{Ce}(V)}\rangle_{\mathrm{BB}}.$$
The goal of this section is to find an explicit admissible extension for $T_{\mathrm{Ce}(V)}$. Recall the following result of Zhang \cite[Lemma 2.1.2]{Zhang_GScycle}:
\begin{lemma}
    Let \(\mathcal{X}\) be an arithmetic scheme and let \(\bar{L}_1, \dots, \bar{L}_n\) be adelically metrized line bundles on the generic fiber \(X\). Assume that the metrics on \(L_i\) at the infinite places are smooth. Then the functional
\[
\widehat{\mathrm{Ch}}^{\dim \mathcal{X} - n}(\mathcal{X}) \longrightarrow \mathbb{R} : \quad \alpha \mapsto \alpha \cdot c_1(\bar{L}_1) \cdots c_1(\bar{L}_n)
\]is represented by an element in \(\widehat{\mathrm{Ch}}_{\dim \mathcal{X} - n}(\mathcal{X})\) denoted by
\[
c_1(\bar{L}_1) \cdots c_1(\bar{L}_n) \cdot [\mathcal{X}] \in \widehat{\mathrm{Ch}}_{\dim \mathcal{X} - n}(\mathcal{X}).
\]
Moreover, this element has the following restriction to the generic fiber:
\[
c_1(L_1) \cdots c_1(L_n)[X].
\]
\end{lemma}
Enlarge $K$ so that $A$ has semistable reduction over $K$. Pick a projective regular model $\calA$ of $A$ defined over $\calO_K$; such a model exists by K\"unnemann \cite{Kunnemann_ProjRegModel_BBht}. 

We generalize the construction $T_{(\cdot)}$ to horizontal cycles and all positive Beauville degrees, as follows.

Fix integers $0\leq d<g$ and $1\leq s\leq g-d$.
Let $\cZ\subset\cA$ be an integral $(d+1)$-dimensional subscheme, flat over $\calO_K$.
Write $Z$ for its generic fiber and $i:Z\hookrightarrow A$ for the inclusion. Set
\[
T_Z^{(s)}:=\mathrm{pr}_{d+s}(\cF_L(Z))
=\frac{1}{(2d+s)!}p_{2,*}\bigl(p_1^*Z\cdot c_1(Q)^{2d+s}\bigr).
\]
For $s=1$, we retain the notation $T_Z=T_Z^{(1)}$ used above.

Zhang's lemma applies with $\cX=\cZ \times \calA$, $n=2d+s$, and all the line bundles equal to the restrictions of $\bar Q$. This gives an element $$\hat c_1(\bar Q)^{2d+s}\cdot[\cZ\times \calA]\in \hat {\mathrm{Ch}}_{g+1-d-s}(\cZ\times \calA).$$ Then we may consider its push-forward under the map $p_2:\cZ\times \calA\to \calA$; see \cite[§2.1]{Zhang_GScycle}. The construction extends by linearity to the following map.

\begin{defn}
    Define the admissible Fourier transform by \begin{align*}
\hat T^{(s)}_{(\cdot)}:Z^{g-d}_{\mathrm{hor}}(\calA)&\longrightarrow\hat {\mathrm{Ch}}_{g+1-d-s}(\calA)\\
\cZ&\longmapsto\hat T^{(s)}_{\cZ}:=\frac 1 {(2d+s)!}p_{2,*}(\hat c_1(\bar Q)^{2d+s}\cdot[\cZ\times \calA]).
\end{align*}
Here $Z^{g-d}_{\mathrm{hor}}(\calA)$ denotes the group of horizontal cycles
with rational coefficients. For $s=1$, write $\hat T_\cZ=\hat T^{(1)}_\cZ$.
\end{defn}

\begin{rmk}
	K\"unnemann considered a similar construction for abelian varieties with everywhere good reduction, or for more general abelian schemes; see \cite[§6]{Kunnemann_ArithFourier}.
\end{rmk}

We will prove that the arithmetic Chow cycle $\hat T^{(s)}_\cZ$ is an admissible extension of $T_Z^{(s)}$. To illustrate the idea, we first consider an infinite place.
\begin{prop}
    The following assertions hold.
    \begin{itemize}
        \item[(1)] The arithmetic Chow cycle $\hat T^{(s)}_\cZ$ restricts to $T_Z^{(s)}$ on the generic fiber.
        \item[(2)] For every infinite place, the curvature of $\hat T^{(s)}_\cZ$ is $\omega (\hat T^{(s)}_\cZ)=0$. 
    \end{itemize}
\end{prop}

\begin{proof}
    Assertion (1) follows from the statement about the generic fiber in Zhang's lemma.
    
    For (2), take flat coordinates $z_1,\ldots,z_g$ on the first factor
    $A(\CC)$ and $w_1,\ldots,w_g$ on the second factor. Then
    \[
    c_1(\bar Q_\CC)=\sum_{j,k}h_{jk}
    (dz_j\wedge d\bar w_k+dw_j\wedge d\bar z_k).
    \]
    Each term contains one differential from the first factor. Since $Z(\CC)$
    has real dimension $2d$ and $s>0$, the restriction of
    $c_1(\bar Q_\CC)^{2d+s}$ to $Z(\CC)\times A(\CC)$ vanishes.
    Its push-forward is the curvature of $\hat T^{(s)}_\cZ$, which is therefore zero.
\end{proof}

We now prove admissibility at every finite place; the proof is more technical than in the case of an infinite place.

\begin{thm}\label{thm:finite-admissibility}
    For every component of a special fiber of $\calA$, the restriction of $\hat T^{(s)}_\cZ$ is numerically trivial.
\end{thm}
\begin{proof}
    Let $v$ be a prime of $\calO_K$. By de Jong’s alteration theorem \cite[Theorem 8.2]{deJongAlteration}, after a finite extension and an alteration, we may work on a regular projective model \(\mathcal Y\) mapping to a model of \(X=Z
    \times A\), with a strictly semistable special fiber at the place $v$ under consideration. Write \(Y\) for its generic fiber and pull back all the bundles under consideration to \(Y\).

    Let $E\subset \cY_v$ be an irreducible component of the special fiber above $v$, so $\dim E=d+g$. 
    Fix a relatively ample line bundle \(H\) on \(\mathcal Y\), and set \(h=H|_E\). For the ample line bundle \(h\) on \(E\), by \cite[Corollary~3.16]{FulgerLehmann_PosConeCycle}, there is a \textit{norm} on \(N_{g-d-s}(E)_\RR\) such that
    \[
        \|\alpha\|=\deg(h^{g-d-s}\cdot\alpha)
        \quad\text{for every pseudoeffective class }\alpha.
    \]

    The desired assertion is that the element $$c_1(\bar Q)^{2d+s}\in N_{g-d-s}(E)_\RR$$ is identically 0. 
    
    We will prove this assertion by estimating the error term in the limiting process with respect to the norm introduced above.

    Write $\bar L_1=p_{1,Z}^*\bar L|_Z$ and $\bar L_2=p_2^*\bar L$. For each positive integer \(n\), set
    \[
        \bar M_n=n\bar L_1+n^{-1}\bar L_2,
        \qquad
        \bar B_n^\pm=\bar M_n\pm\bar Q.
    \]
    Both \(\bar B_n^\pm\) are nef, since $\bar L$ is nef by \cite[Theorems~6.1.1 and~6.1.3]{YuanZhang_Quasiproj} and
    \[
        \bar M_n\pm\bar Q=\frac1n\bigl((x,a)\mapsto[n]i(x)\pm a\bigr)^*\bar L.
    \]
    Fix \(n\). Choose nef model $\Q$-line bundles \(B_{n,j}^\pm\) approximating $\bar B_n^\pm$ on common refinements (for both choices of $\pm$)
    \[
        \pi_j:\mathcal Y_j\longrightarrow\mathcal Y,
    \]and put
    \[
    Q_{n,j}=\frac{B_{n,j}^+-B_{n,j}^-}{2},
    \qquad
    M_{n,j}=\frac{B_{n,j}^++B_{n,j}^-}{2}.
    \]
    These converge respectively to \(\bar Q\) and \(\bar M_n\).
    As $\cY$ is regular, $E$ is Cartier. Write $E_j:=\pi_j^* E$ for the pullback of $E$ as a Cartier divisor, and set $\pi_{j,E}=\pi_j|_{E_j}$.

    Define the following numerical classes on the fixed variety \(E\):
    \[
        \beta_{n,j}=(\pi_{j,E})_*\bigl(c_1(Q_{n,j})^{2d+s}[E_j]\bigr),
    \]
    and
    \[
    \gamma_{n,j}=(\pi_{j,E})_*\bigl(c_1(M_{n,j})^{2d+s}[E_j]\bigr).
    \]
    Expand the expression of $\beta_{n,j}$ using \(Q_{n,j}=\frac{B_{n,j}^+-B_{n,j}^-}{2}\). Here every mixed product of $B_{n,j}^\pm$ is a product of nef bundles with an effective class \([E_j]\), and is therefore pseudoeffective. So is its push-forward. Removing the signs in this expansion gives exactly \(\gamma_{n,j}\). Consequently, applying the triangle inequality for the norm defined by $h$ and noting that every term in the expansion is pseudoeffective, we get
    \[
    \|\beta_{n,j}\|\le\deg_h(\gamma_{n,j}).
    \]
    We can now estimate everything on the special fiber:
    \[
        \begin{aligned}
        \|\beta_{n,j}\|
        &\le
        \deg\!\left(M_{n,j}^{2d+s}\,\pi_j^*H^{g-d-s}\cdot E_j\right)\\
        &\le
        \deg\!\left(M_{n,j}^{2d+s}\,\pi_j^*H^{g-d-s}\cdot\mathcal Y_{j,v}\right)\\
        &=\deg_Y(M_n^{2d+s} H^{g-d-s}).
\end{aligned}
\]The second inequality uses $0\le E_j\le \cY_{j,v}$ and nefness. The last equality follows from the constancy of the intersection degree in a flat projective family.

Since \(L_1\) factors through \(Z\), we have \(L_1^{d+1}=0\). Therefore
\[
\begin{aligned}
\deg_Y(M_n^{2d+s} H^{g-d-s})
&=\sum_{r=0}^{d}\binom {2d+s}r n^{2r-2d-s}
  \deg_Y(L_1^rL_2^{2d+s-r}H^{g-d-s})\\
&\le \frac{C_s}{n^s},
\end{aligned}
\]where
\[
C_s=\deg_Y((L_1+L_2)^{2d+s} H^{g-d-s})
\]is independent of both \(n\) and \(j\). Here $2r-2d-s\leq -s$ for $r\leq d$, and all intersection degrees in the sum are non-negative.
Thus we have proved the uniform inequality
\[
\|\beta_{n,j}\|\le C_s/n^s.
\]

Finally, for every fixed $n$, the classes $\beta_{n,j}$ converge in
$N_{g-d-s}(E)_\RR$ to the class $c_1(\bar Q)^{2d+s}$
as $j\to\infty$, by the proof of Zhang's lemma~\cite[Lemma 2.1.2]{Zhang_GScycle}
applied to $\cY$ with all line bundles equal to $\bar Q$.
Thus
\[
\|c_1(\bar Q)^{2d+s}\|\leq \frac{C_s}{n^s}.
\]
Since $s>0$, letting $n\to\infty$ and pushing down along the alteration
to the model of $X=Z\times A$, and then along $p_2$, gives the desired assertion.
\end{proof}

We have therefore proved the following:
\begin{cor}\label{cor:admissible-components}
    The arithmetic Chow cycle $\hat T^{(s)}_\cZ$ is an admissible extension of $T_Z^{(s)}$. In other words, the image of $\hat T^{(s)}_{(\cdot)}$ lies in $\hat {\mathrm{Ch}}_{g+1-d-s}(\calA)_{\mathrm{adm}}$, the subgroup consisting of admissible extensions of cycles.
\end{cor}

We now turn to the proof of Corollary~\ref{cor:positive-admissible}. In fact, it is a consequence of the bijectivity of the Fourier--Mukai transform on the generic fiber.

\begin{prop}\label{prop:positive-surjectivity}
    For $0\leq d<g$ and $1\leq s\leq g-d$, the map
    $$T^{(s)}:\mathrm{Ch}^{g-d}(A)\longrightarrow\mathrm{Ch}^{d+s}(A)$$ on the generic fiber has image
    $\mathrm{Ch}^{d+s}_s(A)$ and induces an isomorphism
    \[
    T^{(s)}:\mathrm{Ch}^{g-d}_s(A)\xrightarrow{\ \sim\ }
    \mathrm{Ch}^{d+s}_s(A).
    \]
    Consequently, every class in $\mathrm{Ch}^q_{>0}(A)$ admits an admissible
    extension, proving Corollary~\ref{cor:positive-admissible}.
\end{prop}
\begin{proof}
    By the Beauville decomposition and Fourier inversion, we have an isomorphism
    \[
    \calF_L:\mathrm{Ch}^{g-d}_t(A)\xrightarrow{\ \sim\ }
    \mathrm{Ch}^{d+t}_t(A).
    \]
    Here the polarization induced by $L$ is an isogeny, so the Fourier transform
    remains invertible with rational coefficients. The component
    $\calF_L(z_t)$ has codimension $d+t$. Hence
    \[
    T_Z^{(s)}=\mathrm{pr}_{d+s}\calF_L(Z)=\calF_L(Z_s),
    \]
    which proves the image and isomorphism assertions.

    Now fix $q$ and $z_s\in\mathrm{Ch}^q_s(A)$ with $s>0$.
    If $s>q$, Fourier inversion places its preimage in
    $\mathrm{Ch}^{g-q+s}_s(A)=0$, so $z_s=0$.
    Otherwise, take $d=q-s$. The isomorphism above gives a class
    $w_s\in\mathrm{Ch}^{g-q+s}_s(A)$ such that $T^{(s)}_{w_s}=z_s$.
    Choose a cycle representing $w_s$ and let $\cW_s$ be the corresponding
    linear combination of its horizontal closures in $\calA$.
    By Corollary~\ref{cor:admissible-components}, $\hat T^{(s)}_{\cW_s}$
    is an admissible extension of $z_s$.
    Summing over $s>0$ proves the assertion for $\mathrm{Ch}^q_{>0}(A)$.
\end{proof}

\section{Proof of the main theorem}
First assume $d<g$. After a finite extension of $K$, assume that $A$ has semistable reduction; the normalized heights are unchanged. As in the previous section, fix a projective regular model $\calA/\calO_K$ of $A$, and let $\cV=\sum_i a_i[\cV_i]$ be the linear combination of the Zariski closures of the $V_i$ in $\cA$. Write $\cV^-$ for the corresponding horizontal closure of $[-1]_*V$.\footnote{In fact, K\"unnemann's result \cite[Theorem~3.5(v)]{Kunnemann_ProjRegModel_BBht} implies that we can pick a projective regular model $\cA$ such that $[-1]$ can be extended to an automorphism of $\cA$. So $\cV^-$ is actually the pushforward of $\cV$ under such an extension of $[-1]$.} Applying the results of the previous section term by term to these horizontal cycles, and writing $$\mathrm{Ce}(\cV)=\frac 1 2(\cV-\cV^-),$$ we immediately obtain the following:
\begin{thm}
    $\hat T_{\mathrm{Ce}(\cV)}$ is an admissible extension of $T_{\mathrm{Ce}(V)}$.
\end{thm}

Therefore, \begin{align*}
h_{\mathrm{BB},\cF}(\mathrm{Ce}(V))&=\frac{(-1)^d}{[K:\Q]}\langle\mathrm{Ce}(V),T_{\mathrm{Ce}(V)}\rangle_{\mathrm{BB}}\\
&=\frac{(-1)^d}{[K:\Q]}\hat\deg_\cA(\mathrm{Ce}(\cV)\cdot \hat T_{\mathrm{Ce}(\cV)})
\end{align*}

There are four terms in the expansion of the arithmetic intersection by bilinearity, as $\mathrm{Ce}(\cV)=\frac 1 2(\cV-\cV^-)$. All products below are interpreted term by term, so that $[\cV\times\cV]=\sum_{i,j}a_i a_j[\cV_i\times\cV_j]$ includes the mixed terms.
Before multiplying by $(-1)^d/[K:\Q]$, each term contributes $\frac 1 {4(2d+1)!}\bar Q^{2d+1}\cdot [\cV\times\cV]$. 
For example, the term involving $\cV,\cV$ is
\small$$\frac 1 4\hat\deg_\cA(\cV\cdot \hat T_{\cV})=\frac 1 {4(2d+1)!}\hat \deg([\cA\times\cV]\cdot \hat c_1(\bar Q)^{2d+1}\cdot[\cV\times \cA])= \frac 1 {4(2d+1)!}\bar Q^{2d+1}\cdot [\cV\times\cV].$$
\normalsize
A similar computation applies to the remaining terms, using $([-1]\times \id)^*\bar Q=-\bar Q=(\id\times [-1])^*\bar Q$ and the oddness of $2d+1$.

Therefore, adding the four terms, we obtain
\begin{align*}
h_{\mathrm{BB},\cF}(\mathrm{Ce}(V))
&=\frac{(-1)^d}{[K:\Q]}\hat\deg_\cA(\mathrm{Ce}(\cV)\cdot \hat T_{\mathrm{Ce}(\cV)})\\
&=\frac{(-1)^d}{(2d+1)![K:\Q]}\bar Q^{2d+1}\cdot [\cV\times\cV],
\end{align*}
This is the adelic intersection $(-1)^d\bar Q^{2d+1}\cdot[V\times V]/((2d+1)![K:\Q])$ on the generic fiber. If $d=g$, then $V$ is a multiple of $[A]$ and $\mathrm{Ce}(V)=0$; the intersection on the right also vanishes, since $[-1]\times\id$ preserves $[A\times A]$ and negates $\bar Q^{2g+1}$. This proves Theorem~\ref{thm:main} in all dimensions.

\section{Comparison with the height of Lefschetz type}\label{sec:lefschetz}
Assume $g\geq1$. Write $\mathsf L$ for the Lefschetz operator
$z\mapsto c_1(L)\cdot z$ on $\mathrm{Ch}^*(A)$, and recall that
$h^0(L)=\deg(\mathsf L^g[A])/g!$.

\subsection{Primitive Chow groups}
Beauville defined an $\mathrm{SL}_2$-action on $\mathrm{Ch}^*(A)$, which supplies a Lefschetz decomposition on each piece $\mathrm{Ch}^q_s(A)$ in \cite[Theorem~4.2 and \textsection~5]{Beauville_SL2}, which we briefly recall below. 

If $*$ denotes the Pontryagin product, the lowering operator of the associated $\mathfrak{sl}_2$-action is defined by
\[
\Lambda z=\frac{\mathsf L^{g-1}[A]}{h^0(L)(g-1)!}*z.
\]
Define the primitive part of $\mathrm{Ch}^q_s(A)$ by
\[
P^q_s(A)=\ker\bigl(\Lambda:\mathrm{Ch}^q_s(A)
\longrightarrow\mathrm{Ch}^{q-1}_s(A)\bigr).
\]
By \cite[Proposition~5.2]{Beauville_SL2}, this space is zero if $2q>g+s$;
otherwise, we have
\[
P^q_s(A)=\ker\bigl(\mathsf L^{g+s-2q+1}|_{\mathrm{Ch}^q_s(A)}\bigr).
\]
Any nonzero $u\in P^q_s(A)$ generates an irreducible $\mathrm{SL}_2$-subrepresentation, with basis
$$u,\mathsf Lu,\ldots,\mathsf L^{g+s-2q}u.$$
By \cite[Corollary~5.3]{Beauville_SL2}, we have
\[
\mathrm{Ch}^q_s(A)=
\bigoplus_{j=\max\{0,2q-g-s\}}^q\mathsf L^jP^{q-j}_s(A).
\]
In particular, taking $s=1$, then for $1\leq q\leq g$, every element
$z\in\mathrm{Ch}^q_1(A)$ has a unique decomposition
\begin{equation}\label{eq:primitive-decomposition}
z=\sum_{j=\max\{0,2q-g-1\}}^{q-1}z_j,
\qquad z_j=\mathsf L^j u_j,\quad u_j\in P^{q-j}_1(A).
\end{equation}
Here $P^0_1(A)=0$, so the term for $j=q$ is zero.

We end the discussion of Lefschetz operators with the following well-known lemma, for which we provide a proof for the convenience of readers.

\begin{lemma}\label{lem:lefschetz-self-adjoint}
The Lefschetz operator $\mathsf L$ is self-adjoint for the Beilinson--Bloch
height pairing. More precisely, let $x\in\mathrm{Ch}^a(A)$ and
$y\in\mathrm{Ch}^b(A)$ be homologically trivial classes admitting admissible
extensions on a common projective regular model of $A$. For every integer
$r\geq0$ with $a+b+r=g+1$, we have
\[
\langle\mathsf L^r x,y\rangle_{\mathrm{BB}}
=\langle x,\mathsf L^r y\rangle_{\mathrm{BB}}.
\]
\end{lemma}
\begin{proof}
Choose admissible extensions $\hat x,\hat y$ on a projective regular model
$\calA$, and extend $L$ to a hermitian line bundle $\overline{\mathcal L}$
on $\calA$. Put $\hat\ell=\hat c_1(\overline{\mathcal L})$.
Multiplication by $\hat\ell^r$ preserves admissibility: the curvature remains
zero, and for every vertical arithmetic class $v$,
$(\hat\ell^r\hat x)\cdot v=\hat x\cdot(\hat\ell^r v)=0$,
since $\hat\ell^r v$ is again vertical.
The definition of the height pairing therefore gives
\[
\langle\mathsf L^r x,y\rangle_{\mathrm{BB}}
=\hat\deg_{\calA}(\hat\ell^r\hat x\cdot\hat y)
=\hat\deg_{\calA}(\hat x\cdot\hat\ell^r\hat y)
=\langle x,\mathsf L^r y\rangle_{\mathrm{BB}}.
\]
\end{proof}

\subsection{The two Beilinson--Bloch heights}
For $z\in\mathrm{Ch}^q_1(A)$, our normalization reads
\[
h_{\mathrm{BB},\cF}(z)
=\frac{(-1)^{g-q}}{[K:\Q]}\langle z,\cF_Lz\rangle_{\mathrm{BB}}.
\]
We assume $d<g$ throughout this section. For a cycle $V$ of pure dimension $d$, we write $$V=\sum_s V_s\qquad V_s\in\mathrm{Ch}^{g-d}_s(A).$$ 
\begin{prop}\label{prop:BBFourierV=V1}
	We have $$h_{\mathrm{BB},\cF}(\mathrm{Ce}(V))=h_{\mathrm{BB},\cF}(\mathrm{Ce}(V_1)).$$
\end{prop}
\begin{proof}
	Indeed, functoriality of the height pairing under $[n]$ gives, for
	$x_s\in\mathrm{Ch}^q_s(A)$ and
	$y_t\in\mathrm{Ch}^{g+1-q}_t(A)$ homologically trivial,
	\[
	n^{2q-s}\langle x_s,y_t\rangle_{\mathrm{BB}}=\langle [n]^*x_s,y_t\rangle_{\mathrm{BB}}=\langle x_s,[n]_*y_t\rangle_{\mathrm{BB}}
	=n^{2q-2+t}\langle x_s,y_t\rangle_{\mathrm{BB}}.
	\]
	Consequently the pairing vanishes unless $s+t=2$; see also
	\cite[\textsection~5.3, page 68]{Zhang_GScycle}.
	The component of $\cF_L\mathrm{Ce}(V)$ of codimension $g+1-q$
	is $\cF_L(V_1)$, where $q=g-d$. Therefore, the only nonvanishing term among $\langle V_s,\cF_L(V_1)\rangle_{\mathrm{BB}}$ is the term with $s=1$. Therefore, we obtain
	\[
	h_{\mathrm{BB},\cF}(\mathrm{Ce}(V))
	=\frac{(-1)^d}{[K:\Q]}\langle V_1,\cF_L(V_1)\rangle_{\mathrm{BB}}
	=h_{\mathrm{BB},\cF}(V_1).
	\]\end{proof}

On $\mathrm{Ch}^q_1(A)$, we define
\[
\mathrm{Lef}^q=
\begin{cases}
\mathsf L^{g+1-2q},&2q\leq g+1,\\
\bigl(\mathsf L^{2q-g-1}:\mathrm{Ch}^{g+1-q}_1(A)
\xrightarrow{\sim}\mathrm{Ch}^q_1(A)\bigr)^{-1},&2q>g+1.
\end{cases}
\]
The inverse exists by \cite[Proposition~5.5]{Beauville_SL2}.
We use the following normalization of the Beilinson--Bloch height of Lefschetz type:
\[
h_{\mathrm{BB},\mathrm{Lef}}(z)
:=\frac{(-1)^q}{[K:\Q]}\langle z, \mathrm{Lef}^qz\rangle_{\mathrm{BB}},
\qquad z\in\mathrm{Ch}^q_1(A).
\]
In particular, if $2q\leq g+1$, for primitive $z\in P^q_1(A)$, we have 
$$h_{\mathrm{BB},\mathrm{Lef}}(z)=\frac{(-1)^q}{[K:\Q]}\langle z,\mathsf L^{g+1-2q}z\rangle_{\mathrm{BB}}.$$
Our definition has the same sign convention as in the arithmetic standard conjecture of Hodge type, and the height of Lefschetz type is thus conjecturally non-negative on primitive classes; see for example
\cite[\textsection~5.3]{Zhang_GScycle}.

\begin{prop}\label{prop:fourier-lefschetz}
For the decomposition\begin{equation*}
	z=\sum_{j=\max\{0,2q-g-1\}}^{q-1}z_j,
	\qquad z_j=\mathsf L^j u_j,\quad u_j\in P^{q-j}_1(A).
\end{equation*}as in \eqref{eq:primitive-decomposition}, we have
\begin{align}
h_{\mathrm{BB},\mathrm{Lef}}(z)
&=\sum_j h_{\mathrm{BB},\mathrm{Lef}}(z_j),\notag\\
h_{\mathrm{BB},\cF}(z)
&=h^0(L)\sum_j(-1)^j\frac{j!}{(g+1-2q+j)!}
h_{\mathrm{BB},\mathrm{Lef}}(z_j).\label{eq:fourier-lefschetz}
\end{align}
In particular, if $2q\leq g+1$ and $z\in P^q_1(A)$ is primitive, then
\[
h_{\mathrm{BB},\cF}(z)
=\frac{h^0(L)}{(g+1-2q)!}\,
h_{\mathrm{BB},\mathrm{Lef}}(z).
\]
\end{prop}
\begin{proof}
Put $m=g+1-2q$, and let $F_{\mathrm B}$ be Beauville's normalized Fourier transform.
His kernel class is $-c_1(Q)$ and his transform is divided by
$h^0(L)$ \cite[\textsection~1.6 and \textsection~4.1]{Beauville_SL2}. Therefore
\[
F_{\mathrm B}=\frac{1}{h^0(L)}[-1]^*\cF_L,
\qquad
\cF_L=-h^0(L)F_{\mathrm B}\quad\text{on Beauville degree }1.
\]
By \cite[Corollary~5.4]{Beauville_SL2}, for every $\max\{0,2q-g-1\}\leq j\leq q-1$,
\[
\cF_L(\mathsf L^j u_j)
=h^0(L)(-1)^{m+j+1}\frac{j!}{(m+j)!}
\mathsf L^{m+j}u_j,
\qquad
\mathrm{Lef}^q(\mathsf L^j u_j)=\mathsf L^{m+j}u_j.
\]
All exponents and factorials here are nonnegative by the range of $j$.
Since $z_j$ has codimension $q$, we find that
\[
\begin{aligned}
h_{\mathrm{BB},\cF}(z_j)
&=\frac{(-1)^{g-q}}{[K:\Q]}\langle z_j,\cF_L(z_j)\rangle_{\mathrm{BB}}\\
&=\frac{h^0(L)(-1)^{g-q+m+j+1}}{[K:\Q]}\frac{j!}{(m+j)!}
\langle z_j,\mathrm{Lef}^q z_j\rangle_{\mathrm{BB}}\\
&=h^0(L)(-1)^j\frac{j!}{(m+j)!}
h_{\mathrm{BB},\mathrm{Lef}}(z_j),
\end{aligned}
\]
where we used $(-1)^{g-q+m+j+1+q}=(-1)^j$.

Finally we check that the summands are orthogonal for the Lefschetz pairing. If $i>j$, the
self-adjointness of $\mathsf L$ for the Beilinson--Bloch height pairing
(see Lemma~\ref{lem:lefschetz-self-adjoint}) gives
\[
\langle\mathsf L^i u_i,\mathsf L^{m+j}u_j\rangle_{\mathrm{BB}}
=\langle u_i,\mathsf L^{m+i+j}u_j\rangle_{\mathrm{BB}}=0,
\]
since $m+i+j>m+2j$ and $\mathsf L^{m+2j+1}u_j=0$ by primitivity.
If $i<j$, similar arguments give
\[
\langle\mathsf L^i u_i,\mathsf L^{m+j}u_j\rangle_{\mathrm{BB}}
=\langle\mathsf L^{m+i+j}u_i,u_j\rangle_{\mathrm{BB}}=0,
\]
since $m+i+j>m+2i$ and $\mathsf L^{m+2i+1}u_i=0$.
Since $\cF_L$ is a scalar multiple of $\mathrm{Lef}^q$ on each summand $\mathsf L^jP^{q-j}_1(A)$, the orthogonality for the height of Lefschetz type gives the same orthogonality for the height of Fourier type.
\end{proof}

As an application, we prove the positivity of Beilinson--Bloch heights of Fourier type follows from the arithmetic standard conjecture of Hodge type.
\begin{cor}
	\label{cor:fourier-height-sign}
	The arithmetic standard conjecture of Hodge type implies Conjecture~\ref{Conj:FourierPos}.
\end{cor}

\begin{proof}
For $1\leq r\leq(g+1)/2$, the arithmetic standard conjecture of Hodge type predicts that
\[
h_{\mathrm{BB},\mathrm{Lef}}(u)
=\frac{(-1)^r}{[K:\Q]}\langle u,\mathsf L^{g+1-2r}u\rangle_{\mathrm{BB}}\geq0
\qquad\text{for }u\in P^r_1(A);
\]
see \cite[\textsection~5.3]{Zhang_GScycle}.
For the decomposition~\eqref{eq:primitive-decomposition}, with
$z_j=\mathsf L^ju_j$ and $u_j\in P^{q-j}_1(A)$, self-adjointness gives
\[
h_{\mathrm{BB},\mathrm{Lef}}(z_j)
=\frac{(-1)^q}{[K:\Q]}\langle u_j,\mathsf L^{g+1-2q+2j}u_j\rangle_{\mathrm{BB}}
=(-1)^j h_{\mathrm{BB},\mathrm{Lef}}(u_j).
\]
Thus Proposition~\ref{prop:fourier-lefschetz} yields
\[
h_{\mathrm{BB},\cF}(z)
=h^0(L)\sum_j\frac{j!}{(g+1-2q+j)!}
h_{\mathrm{BB},\mathrm{Lef}}(u_j).
\]
All coefficients in this sum are positive, so the conjecture predicts that
$h_{\mathrm{BB},\cF}(z)\geq0$ for $z\in\mathrm{Ch}^q_1(A)$.

Taking $z=V_1$, applying Proposition~\ref{prop:BBFourierV=V1}, we obtain the conclusion of Conjecture~\ref{Conj:FourierPos}.
This explains the factor $(-1)^d$ in our normalization, where $d=g-q$ is the dimension of the cycle.
\end{proof}

\bibliographystyle{amsalpha}

\bibliography{references}
\end{document}